\documentclass[11pt, leqno, twoside]{article}   
\usepackage[applemac]{inputenc}
\usepackage[T1]{fontenc}
\usepackage{mathtools, bm}
\usepackage{amsthm}
\usepackage{amssymb, bm}
\usepackage{bbm}
\usepackage[mathscr]{euscript}
\usepackage[margin=1in]{geometry}
\usepackage{systeme}
\usepackage[french]{babel}
\usepackage{chngcntr}
\usepackage{titling}
\usepackage{cancel}
\usepackage{stmaryrd}
\usepackage{siunitx}
\usepackage{fancyhdr, ifthen}
\usepackage[normalem]{ulem}
\usepackage{scrextend} 

\addto\captionsfrench{%
}

\newenvironment{poliabstract}[1]
  {\renewcommand{\abstractname}{#1}\begin{abstract}}
  {\end{abstract}}

\newtheoremstyle{theoremnoperiod}
  {\topsep}   
  {\topsep}   
  {\normalfont}  
  {0pt}       
  {\bfseries} 
  {}          
  {5pt plus 1pt minus 1pt} 
  {}          

\newtheorem*{theorem*}{Théorème}
\newtheorem{theorem}{Théorème}[section]

\newtheorem{lemma}[theorem]{Lemme}
\newtheorem*{conjecture*}{Conjecture}

\numberwithin{equation}{section}

\theoremstyle{theoremnoperiod}
\newtheorem*{thmnodot*}{Théorème}
\newtheorem{thmnodot}[theorem]{Théorème}
\newtheorem{lemmanodot}[theorem]{Lemme}

\DeclareMathOperator{\li}{li}

\DeclareMathOperator{\R}{\mathbb{R}}

\DeclareMathOperator{\C}{\mathbb{C}}

\DeclareMathOperator{\sP}{\mathscr{P}}
\DeclareMathOperator{\sE}{\mathscr{E}}
\DeclareMathOperator{\sB}{\mathscr{B}}
\DeclareMathOperator{\sS}{\mathscr{S}}

\DeclareMathOperator{\sI}{\mathscr{I}}

\def\HH{{\mathscr H}}

\DeclareMathOperator{\e}{\rm e}
\def\d{\,{\rm d}}
\def\1{{\bf 1}}
\DeclareMathOperator{\PP}{\mathbb P}

\DeclareMathOperator{\sA}{\mathscr A}

\DeclareMathOperator{\cA}{\mathcal A}

\def\p{{p_{m,\nu}}}

\renewcommand{\leq}{\leqslant}

\renewcommand{\geq}{\geqslant}

\usepackage{color}
\definecolor{vert}{rgb}{0,0.5,0}
\definecolor{violet}{rgb}{0.7,0.1,0.8}

\title{Étude statistique du facteur premier médian, 3 : \goodbreak lois de répartition}
\author{Jonathan Rotgé\thanks{Adresse e-mail : jonathan.rotge@etu.univ-amu.fr}\\ \\ {\it\small  Université d'Aix-Marseille,\, Institut de Mathématiques de Marseille CNRS UMR 7373,}\\ {\it\small 163 Avenue De Luminy, Case 907, 13288 Marseille Cedex 9, FRANCE}}
\date{}

\begin{document}
\maketitle
\thispagestyle{empty}

\selectlanguage{english}
\begin{poliabstract}{Abstract} 
We consider the  Gaussian limit law for the distribution of the middle prime factor of an integer, defined according to multiplicity or not. We obtain an optimal bound for the speed of convergence, thereby improving on previous estimates available in the literature. 
 \end{poliabstract}

\selectlanguage{french}
\begin{poliabstract}{Résumé}
   Nous nous intéressons à l'approximation gaussienne de la répartition du facteur premier médian d'un entier, défini en tenant compte ou non, de la multiplicité. Nous obtenons une majoration optimale de la vitesse de convergence, améliorant ainsi les résultats existant dans la littérature.
\end{poliabstract}
\bigskip


\section{Introduction et énoncé du résultat}
Pour tout entier naturel $n\geq 2$, posons
\[\omega(n):=\sum_{p|n}1,\qquad\Omega(n):=\sum_{p^k||n}k,\]
et notons $\nu\in\{\omega,\Omega\}$ l'une ou l'autre de ces fonctions. Si $\{q_j(n)\}_{1\leq j\leq\omega(n)}$ désigne la suite croissante des facteurs premiers de $n$ comptés sans multiplicité et $\{Q_j(n)\}_{1\leq j\leq\Omega(n)}$ celle des facteurs premiers de $n$ comptés avec multiplicité, nous écrivons
\[\p(n):=\begin{cases}q_{\lceil\omega(n)/2\rceil}(n)&{\rm si\;} \nu=\omega\\ Q_{\lceil\Omega(n)/2\rceil}(n)&{\rm si\;} \nu=\Omega\end{cases},\qquad P^-(n):=q_1(n),\qquad P^+(n):=q_{\omega(n)}(n).\]
Dans toute la suite, les lettres $p$ et $q$ désignent des nombres premiers.
Posons
\[\Phi(v):=\frac1{\sqrt{2\pi}}\int_{-\infty}^v\e^{-t^2/2}\d t\quad(v\in\R),\]
la fonction de répartition de la loi normale.\par
Nous dirons qu'une fonction arithmétique $f$ est d'ordre normal $g$ si $g$ est une fonction arithmétique telle que, pour tout $\varepsilon>0$, on ait
\[\lim_{x\to\infty}\frac1x\bigg|\bigg\{n\leq x:|f(n)-g(n)|\leq\varepsilon g(n)\bigg\}\bigg|=1.\]

À l'aide d'une inégalité de type Turán-Kubilius, De Koninck et Kátai \cite{bib:dekoninck:normalorder} ont établi que l'ordre normal de la fonction arithmétique $\log_2\p(n)$ est $\tfrac12\log_2 x$.\footnote{Ici et dans la suite nous notons $\log_k$ la $k$-ième itérée de la fonction logarithme.}  Par la suite, De Koninck, Doyon et Ouellet \cite{bib:dekoninck:gaussian} ont précisé ce résultat en mettant en évidence un comportement gaussien. Posant
\[\cA_\nu(x,t):=\frac1x\Big|\Big\{n\leq x:\log_2\p(n)-\tfrac12\log_2 x<t\sqrt{\log_2 x}\Big\}\Big|\quad(x\geq 3,\,t\in\R),\]
ils obtiennent le résultat suivant. 
\begin{thmnodot*}{\bf (\cite[th. 1]{bib:dekoninck:gaussian}).}
	Soit $0<\varepsilon<\tfrac18$ fixé. Pour tout réel $t$ vérifiant $|t|\ll(\log_2 x)^{1/8-\varepsilon}$, nous avons
	\begin{equation}\label{eq:freq:KoDoOu}
		\cA_\nu(x,t)=\Phi(2t)+O\bigg(\frac1{\sqrt{\log_3 x}}\bigg).
	\end{equation}
\end{thmnodot*}
En 2023, McNew, Pollack et Singha Roy \cite{bib:pollack} établissent une version uniforme en $t$ de \eqref{eq:freq:KoDoOu}, avec un terme d'erreur plus précis, soit $O((\log_3 x)^{3/2}/\sqrt{\log_2 x})$.\par 
Nous nous proposons ici de fournir une version optimale de l'approximation \eqref{eq:freq:KoDoOu}. L'étude des lois locales menée dans \cite{bib:papier2} constitue l'ingrédient essentiel de la preuve.\par
Nous démontrons le résultat suivant.

\begin{theorem}\label{thm:gaussiandistrib}
	Nous avons uniformément
	\begin{equation}\label{eq:gaussiandistrib:uniform}
		\cA_\nu(x,t)=\Phi(2t)+O\bigg(\frac1{\sqrt{\log_2 x}}\bigg)\quad(x\geq 3,\,t\in\R).
	\end{equation}
	De plus, le terme d'erreur est optimal.
\end{theorem}


\section{Lois locales du facteur premier médian}

Notons
\[\kappa:=\lim_{n\to\infty}\bigg(\sum_{q\leq n}\frac1q-\log_2 n\bigg)\]
la constante de Meissel-Mertens et $\gamma$ la constante d'Euler-Mascheroni. Définissons alors 
\begin{align*}
	\HH_\nu(z)&:=\begin{cases}
		\displaystyle\e^{\kappa}\prod_{q}\Big(1+\frac{z}{q-1}\Big)\e^{-z/q}\quad(z\in\C)\qquad&\textnormal{si }\nu=\omega,\\
		\displaystyle\e^{\gamma z}\prod_{q}\Big(1-\frac{1}{q}\Big)^z\Big(1-\frac{z}{q}\Big)^{-1}\quad(\Re z<2)&\textnormal{si }\nu=\Omega.
	\end{cases}
\end{align*}
Notons $a_\nu:=0$ si $\nu=\omega$, $a_\nu:=\tfrac15$ si $\nu=\Omega$, définissons
\begin{equation}\label{def:fnu;rho;rhonu}
	\begin{gathered}
		f_\nu(z):=\frac{\HH_\nu(z)\e^{-\gamma/z}}{\Gamma(1+1/z)}\ (0<\Re z<2),\ \varrho_\nu(v):=\frac{(1+w)f_\nu(w)}{2w\sqrt{\pi vw}}\quad \bigg(a_\nu<v<1,\,w:=\sqrt{\frac{1-v}v}\bigg)
	\end{gathered}
\end{equation}
et posons
\[\beta_p=\beta_p(x):=\frac{\log_2 p}{\log_2 x},\qquad\varepsilon_x:=\frac1{\log_2 x}\quad(3\leq p\leq x).\]
Notons d'emblée que 
\[p=\e^{(\log x)^{\beta_p}}\quad (3\leqslant p\leqslant x).\]
Les lois locales de la répartition de $\p(n)$ dans $[1,x]$ sont données par les quantités
\[M_\nu(x,p):=|\{n\leq x:\p(n)=p\}|\quad(3\leq p\leq x).\]
L'énoncé suivant fournit une estimation de la quantité $M_{\nu}(x,p)$ pour de grandes valeurs de $p$.

\begin{thmnodot}{\bf (\cite[th. 1.1]{bib:papier2}).}\label{th:eq:Mp}
	Soit $\varepsilon>0$. Sous la condition $\tfrac15+\varepsilon<\beta_p<1-\varepsilon$, nous avons uniformément
	\begin{equation}\label{eq:Mp:pomega}
		M_\nu(x,p)=\frac{\{1+O(\varepsilon_x)\}\varrho_\nu(\beta_p)x}{p(\log x)^{1-2\sqrt{\beta_p(1-\beta_p)}}\sqrt{\log_2 x}}.
	\end{equation}
\end{thmnodot}

Pour tout ensemble de nombre premiers non vide $E$, posons
\[\omega(n,E):=\sum_{p|n,\,p\in E}1\quad(n\geq 1),\quad\Omega(n,E):=\sum_{p^k\|n,\,p\in E}k\quad (n\geq 1),\quad E(x):=\sum_{p\leq x,\;p\in E}\frac{1}{p}\quad (x\geq 2).\]
Dans la suite, nous utiliserons la notation $\nu(n,E)$ pour faire simultanément référence à $\omega(n,E)$ ou $\Omega(n,E)$. Définissons enfin
\[Q(v):=v\log v-v+1\quad(v>0).\]


\begin{lemmanodot}{\bf (\cite[th.\thinspace08, th.\thinspace09]{hall_tenenbaum}). }\label{hall_tenenbaum_resultat} Soit $p_0$ un nombre premier et $E$ un ensemble de nombres premiers non vide tel que $\min\{p:p\in E\}\geq p_0$. Pour tous $0<a<1<b<p_0$, nous avons
 	\begin{equation}\label{Omega_n_E}
 		\sum_{\substack{n\leq x\\\nu(n,E)\leq a E(x)}}1\ll_{a,p_0} x\e^{-E(x)Q(a)},\qquad\sum_{\substack{n\leq x\\\nu(n,E)\geq b E(x)}}1\ll_{b,p_0} x\e^{-E(x)Q(b)}.
	 \end{equation}
\end{lemmanodot}

Posons
\begin{equation}\label{def:eta_x}
	\eta_x:=\sqrt{\frac{\log_3 x}{\log_2 x}}\quad(x\geq 16).
\end{equation}
Nous ferons usage du résultat suivant.

\begin{lemma}\label{l:majo:Mnu:uniform}
	Sous la conditon $|\beta_p-\tfrac12|\geq\eta_x$, nous avons la majoration uniforme
	\begin{equation}\label{eq:majo:Mnu:uniform}
		M_{\nu}(x,p)\ll\frac{x}{p(\log_2 x)^{5/2}}\quad(3\leq p\leq x).
	\end{equation}
\end{lemma}
\begin{proof}
	Posons
	\[\gamma_\nu(v):=\begin{cases}
		\tfrac12(1-3v)\qquad&\textnormal{si }0<v\leq a_\nu,\\
		1-2\sqrt{v(1-v)}&\textnormal{si }a_\nu\leq v<1.
	\end{cases}\]
	Soit $0<\varepsilon<\tfrac1{17}$. Nous distinguons selon les valeurs de $\beta_p$.\par
	Considérons le cas $\varepsilon<\beta_p\leq\tfrac15$. Puisque 
	\[\gamma_\nu(v)\geq\tfrac12(1-3v)\geq\tfrac15\quad(v\leq\tfrac15),\]
	nous déduisons directement de \cite[th. 1.1]{bib:papier2} que
	\begin{equation}\label{eq:majo:Mnu:uniform:smallp}
		M_\nu(x,p)\ll \frac x{p(\log x)^{1/5}}\quad (\varepsilon\leq\beta_p\leq\tfrac15).
	\end{equation}\par
	Lorsque $\tfrac15\leq\beta_p\leq1-\varepsilon$, nous avons
	\[\gamma_\nu(\beta_p)=1-2\sqrt{\beta_p(1-\beta_p)}\geq\gamma_\nu(\tfrac12-\eta_x)=\gamma_\nu(\tfrac12+\eta_x).\]
	Un développement de Taylor à l'ordre $4$ fournit
	\[\gamma_\nu(\tfrac12+v)=2v^2+O(v^4)\quad(v\ll 1).\]
	Une nouvelle application de \cite[th. 1.1]{bib:papier2} permet alors d'obtenir
	\begin{equation}\label{eq:majo:Mnu:uniform:bigp}
		M_\nu(x,p)\ll\frac x{p(\log x)^{2\eta_x^2+O(\eta_x^4)}\sqrt{\log_2 x}}\ll\frac x{p(\log_2 x)^{5/2}}\quad(\tfrac15\leq\beta_p\leq1-\varepsilon,\, |\beta_p-\tfrac12|\geq\eta_x).
	\end{equation}
	\par Il reste à traiter le cas des valeurs extrêmes de $\beta_p$. Lorsque $\beta_p<\varepsilon$, nous distinguons deux cas selon les valeurs de $\nu(n)$. Si $\nu(n)\leq\tfrac14\log_2 x$, la première formule \eqref{Omega_n_E} fournit la majoration
	\begin{equation}\label{eq:majo:smallbeta:smallk}
		\sum_{\substack{n\leq x,\,p|n\\ \nu(n,\PP)\leq(\log_2 x)/4}}1\leq\sum_{\substack{d\leq x/p\\ \omega(d,\PP)\leq\{1+\PP(x/p)\}/4}}1\ll\frac{x}{p(\log x)^{Q(1/4)}}.	\end{equation}
	\begin{sloppypar}
	\noindent Si $\nu(n)\geq\tfrac14\log_2 x$, alors $n$ possède au moins $(\log_2 x)/8$ facteurs premiers dans l'intervalle $[3,\exp\{(\log x)^\varepsilon\}]$. Posons $\PP_1:=[3,\exp\{(\log x)^\varepsilon\}]$. Le théorème de Mertens permet d'obtenir, pour $x$ assez grand, la majoration
	\end{sloppypar}
	\[\PP_1\Big(\frac xp\Big)=\varepsilon\log_2 x+O(1)<2\varepsilon\log_2 x,\]
	de sorte qu'une application de la deuxième formule \eqref{Omega_n_E} fournit
	\begin{equation}\label{eq:majo:smallbeta:bigk}
		\sum_{\substack{n\leq x,\, p|n\\ \nu(n,\PP_1)\geq(\log_2 x)/8}}1\leq\sum_{\substack{d\leq x/p\\ \Omega(d,\PP_1)\geq \PP_1(x/p)/16\varepsilon}}1\ll\frac{x}{p(\log x)^{Q(1/16\varepsilon)}},
	\end{equation}
	puisque $1/16\varepsilon\geq\tfrac{17}{16}>1$.\par
	Enfin, le cas $\beta_p>1-\varepsilon$ relève d'un traitement analogue au précédent.\par
	La majoration \eqref{eq:majo:Mnu:uniform} est alors une conséquence de \eqref{eq:majo:Mnu:uniform:smallp}, \eqref{eq:majo:Mnu:uniform:bigp}, \eqref{eq:majo:smallbeta:smallk} et \eqref{eq:majo:smallbeta:bigk}.
\end{proof}

\section{Preuve du Théorème \ref{thm:gaussiandistrib}}
	Rappelons la définition de $\cA_\nu$ en \eqref{thm:gaussiandistrib}. Posons 
	\[\lambda(x,t):=\tfrac12\log_2 x+t\sqrt{\log_2 x},\quad \sP_{x,t}:=[3,\exp\big(\e^{\lambda(x,t)}\big)\big[\quad(x\geq 3,\,t\in\R),\]
	de sorte que
	\begin{align*}
		\cA_\nu(x,t)&=\frac1x\sum_{\substack{n\leq x\\ \log_2 \p(n)<\lambda(x,t)}}1=\frac1x\sum_{\substack{p\leq x\\ \log_2 p<\lambda(x,t)}}\sum_{\substack{n\leq x\\ 			\p(n)=p}}1\\
		&=\frac1x\sum_{\substack{p\leq x\\ p<\exp(\e^{\lambda(x,t)})}}M_\nu(x,p)=\frac1x\sum_{p\in\sP_{x,t}}M_\nu(x,p).
	\end{align*}
	Définissons
	\[\sP_{x,t}^+:=\Big[\exp\Big(\sqrt{\log x}\e^{-\sqrt{\log_3 x}}\Big),\exp\big(\e^{\lambda(x,t)}\big)\Big[,\quad \sP_{x,t}^-:=\sP_{x,t}\smallsetminus\sP_{x,t}^+\quad(x\geq 	16,\, t\in\R)\]
	et remarquons d'emblée que $\sP_{x,t}^+=\varnothing$ dès lors que $t<-\sqrt{\log_3 x}$, et donc $\sP_{x,t}^-=\sP_{x,t}$. Notons également que
	\[p\in\sP_{x,t}^+\Leftrightarrow \tfrac12-\eta_x\leq\beta_p<\tfrac12+t\sqrt{\varepsilon_x},\]
	de sorte que, d'après \eqref{eq:Mp:pomega}, nous pouvons écrire
	\begin{equation}\label{eq:defcPnu;cEnu}
		\sA_\nu(x,t)=\frac1x\bigg\{\sum_{p\in\sP_{x,t}^+}\frac{\{1+O(\varepsilon_x)\}\varrho_\nu(\beta_p)x}{p(\log x)^{1-2\sqrt{\beta_p(1-\beta_p)}}\sqrt{\log_2 x}}+	\sum_{p\in\sP_{x,t}^-}M_\nu(x,p)\bigg\}=:\sS_\nu(x,t)+\sE_\nu(x,t).
	\end{equation}
Rappelons la définition de $\varrho_\nu$ en \eqref{def:fnu;rho;rhonu} et posons
	\begin{equation}\label{eq:def:kappa;gnux;cPnustar}
		\begin{gathered}
			\kappa(v):=2\sqrt{v(1-v)}-1\ (0<v<1),\quad g_{\nu,x}(p):=\frac{\varrho_\nu(\beta_p)(\log x)^{\kappa(\beta_p)}}p\ (x\geq 3,\,a_\nu<\beta_p<1),\\
			\sS_\nu^*(x,t):=\sum_{p\in\sP_{x,t}^+}g_{\nu,x}(p)\quad(x\geq 16,\ t\in\R).
		\end{gathered}
	\end{equation}
	Réécrivons $\sS_\nu^*$ sous la forme d'une intégrale de Stieltjes. Il vient
	\begin{equation}\label{eq:evalcPstar:decomp}
		\begin{aligned}
			\sS_\nu^*(x,t)&=\int_{\sP_{x,t}^+}g_{\nu,x}(v)\d\pi(v)=\int_{\sP_{x,t}^+}g_{\nu,x}(v)\d\li(v)+\int_{\sP_{x,t}^+}g_{\nu,x}(v)\d\{\pi(v)-\li(v)\}\\
			&=:\sI_1(x,t)+\sI_2(x,t).
		\end{aligned}
	\end{equation}\par
	Nous traitons $\sI_2$ comme un terme d'erreur. Une forme forte du théorème des nombres premiers fournit l'existence d'une constante $c>0$ telle que $\pi(t)-\li(t)\ll t\e^{-c\sqrt{\log t}}\ (t\geq 2)$ (voir {\it e.g.} \cite[Théorème II.4.1]{bib:tenenbaum}). Une intégration par parties implique
\begin{align}\label{eval_J2_intermediaire}
	\sI_2(x,t)=\Big[g_{\nu,x}(v)(\pi(v)-\li(v))\Big]_{\sP_{x,t}^+}-\int_{\sP_{x,t}^+}g'_{\nu,x}(v)\{\pi(v)-\li(v)\}\d v.
\end{align}
D'une part, le terme entre crochets peut être majoré par
\begin{equation}\label{eq:evalcPstar:cP21}
	(\log x)^{\kappa\big(\beta_{\exp\{\sqrt{\log x}\exp({-\sqrt{\log_3 x}})}\big)}\exp\Big(-c(\log x)^{1/4}\e^{-\sqrt{\log_3 x}/2}\Big)\ll\e^{-(\log x)^{1/5}}.
\end{equation}
D'autre part, puisque $\log g_{\nu,x}(v)=\log\varrho_\nu(\beta_v)+\kappa(\beta_v)\log_2 x-\log v$, nous pouvons écrire
\begin{align*}
	\frac{g'_{\nu,x}(v)}{g_{\nu,x}(v)}=O\Big(\frac{1}{v\log v}\Big)-\frac{2\log_2 v}{v\log v\log_2 x}+O\Big(\frac{1}{v\log v}\Big)-\frac{1}{v}=-\frac1v+O\Big(\frac{1}{v\log v}\Big).
\end{align*}
Ainsi $g'_{\nu,x}(v)<0$ pour $x$ assez grand, donc $g_{\nu,x}$ est décroissante sur $\sP_{x,t}^+$. Une nouvelle intégration par parties permet alors de montrer que l'intégrale de \eqref{eval_J2_intermediaire} est
\begin{equation}\label{eq:evalcPstar:cP22}
	\ll \int_{\sP_{x,t}^+}g_{\nu,x}(v)\e^{-c\sqrt{\log v}}\d v+\e^{-(\log x)^{1/5}}\ll \e^{-(\log x)^{1/5}}.
\end{equation}
De \eqref{eq:evalcPstar:cP21} et \eqref{eq:evalcPstar:cP22}, nous déduisons que 
\begin{equation}\label{eq:evalcPstar:cP2}
	\sI_2(x,t)\ll\e^{-(\log x)^{1/5}}.
\end{equation}
\par 
Enfin, nous avons
\begin{equation}\label{eq:evalcPstar:cP1}
	\sI_1(x,t)=\int_{\sP_{x,t}^+}\frac{g_{\nu,x}(v)}{\log v}\d v,
\end{equation}
de sorte que, d'après \eqref{eq:evalcPstar:decomp}, \eqref{eq:evalcPstar:cP2} et \eqref{eq:evalcPstar:cP1}, nous obtenons
\begin{equation}\label{eq:eval:cPstar}
	\sS_\nu^*(x,t)=\int_{\sP_{x,t}^+}\frac{g_{\nu,x}(v)}{\log v}\d v+O\big(\e^{-(\log x)^{1/5}}\big).
\end{equation}
Rappelons la définition de $\eta_x$ en \eqref{def:eta_x} et posons
\[\sB^+_{x,t}:=\big[\tfrac12-\eta_x,\tfrac12+t\sqrt{\varepsilon_x}\big[\quad(x\geq 16,\, t\in\R),\quad h_{\nu,x}(v):=\varrho_\nu(v)(\log x)^{\kappa(v)}\quad(a_\nu<v<1).\]
	Remarquons que $\sB^+_{x,t}=\varnothing$ dès lors que $t<-\eta_x$. Rappelons les définitions de $\sS_\nu$ en \eqref{eq:defcPnu;cEnu} et de $\sS_\nu^*$ en \eqref{eq:def:kappa;gnux;cPnustar}. Avec le changement de variables $v=\e^{(\log x)^\beta}$, nous obtenons, d'après \eqref{eq:eval:cPstar},
	\begin{equation}\label{eq:sPnu:inter}
		\sS_\nu(x,t)=\{1+O(\varepsilon_x)\}\sqrt{\log_2 x}\int_{\sB_{x,t}^+}h_{\nu,x}(\beta)\d\beta.
	\end{equation}
	\par Nous évaluons $\sS_\nu$ en distinguant plusieurs plages de valeurs de  $t$.\par
	Considérons dans un premier temps le cas $t\in\big[-\sqrt{\log_3 x},\sqrt{\log_3 x}\big]$. La fonction $\kappa$ atteint son maximum $0$ en $\beta_0=\tfrac12$. Un développement de Taylor à l'ordre $4$ fournit
	\begin{equation}\label{eq:taylor:kappa}
		\kappa(\beta_0+v)=-2v^2+O(v^4)\quad(v\ll 1).
	\end{equation}
	Un second développement de Taylor à l'ordre $1$ fournit
	\begin{equation}\label{eq:taylor:varrho}
		\frac{\varrho_\nu(\beta_0+v)}{\varrho_\nu(\beta_0)}=1+O(v)\quad(v\ll 1).
	\end{equation}
	Les formules \eqref{eq:taylor:kappa} et \eqref{eq:taylor:varrho} fournissent le 	développement
	\begin{equation}\label{eq:taylor:hnu}
		\frac{h_{\nu,x}(\beta_0+v)}{h_{\nu,x}(\beta_0)}=\{1+O(v+v^4\log_2 x)\}\e^{-2v^2\log_2 x}\quad(x\geq 3,\,v\ll 1),
	\end{equation}
	de sorte que
	\begin{equation}\label{eq:integral:rapportgnux:decomp}
		\int_{-\eta_x}^{t\sqrt{\varepsilon_x}}\frac{h_{\nu,x}(\beta_0+v)}{h_{\nu,x}(\beta_0)}\d v=\int_{-\eta_x}^{t\sqrt{\varepsilon_x}}\e^{-2v^2\log_2 x}\d v+O\bigg(\int_{-\eta_x}^{t\sqrt{\varepsilon_x}}\{|v|+v^4\log_2 x\}\e^{-2v^2\log_2 x}\d v\bigg).
	\end{equation}
	Évaluons l'intégrale principale du membre de droite de \eqref{eq:integral:rapportgnux:decomp}. Le changement de variables $u=2v\sqrt{\log_2 x}$ fournit
	\[\int_{-\eta_x}^{t\sqrt{\varepsilon_x}}\e^{-2v^2\log_2 x}\d v=\frac{1}{2\sqrt{\log_2 x}}\int_{-2\sqrt{\log_3 x}}^{2t}\e^{-u^2/2}\d u.\]
	Puisque, par ailleurs
	\begin{equation}\label{eq:majo:gaussian:tail}
		\int_{-\infty}^{-2\sqrt{\log_3 x}}\e^{-u^2/2}\d u\ll\frac{\e^{-2\log_3 x}}{\sqrt{\log_3 x}}\ll\varepsilon_x^2,
	\end{equation}
	nous obtenons
	\begin{equation}\label{eq:int:rapportgnux:main}
		\int_{-\eta_x}^{t\sqrt{\varepsilon_x}}\e^{-2v^2\log_2 x}\d v=\frac{\{1+O(\varepsilon_x^2)\}\sqrt\pi\Phi(2t)}{\sqrt{2\log_2 x}}.
	\end{equation}
	Évaluons désormais le terme d'erreur de \eqref{eq:integral:rapportgnux:decomp}. Nous avons d'une part
	\begin{equation}\label{eq:int:rapportgnux:error1}
		\int_{-\eta_x}^{t\sqrt{\varepsilon_x}}|v|\e^{-2v^2\log_2 x}\d v\leq\int_{-\eta_x}^{|t|\sqrt{\varepsilon_x}}|v|\e^{-2v^2\log_2 x}\d v\ll\varepsilon_x,
	\end{equation}
	uniformément par rapport à $t$, et d'autre part,
	\begin{equation}\label{eq:int:rapportgnux:error2}
		(\log_2 x)\int_{-\eta_x}^{t\sqrt{\varepsilon_x}}v^4\e^{-2v^2\log_2 x}\d v\ll\{1+t^3\e^{-2t^2}\}\varepsilon_x^{3/2}\ll\varepsilon_x^{3/2}.
	\end{equation}
	En regroupant, les estimations \eqref{eq:int:rapportgnux:main}, \eqref{eq:int:rapportgnux:error1} et \eqref{eq:int:rapportgnux:error2}, nous obtenons, d'après \eqref{eq:integral:rapportgnux:decomp},
	\begin{equation}\label{eq:integral:rapportgnux:final}
		\int_{-\eta_x}^{t\sqrt{\varepsilon_x}}\frac{h_{\nu,x}(\beta_0+v)}{h_{\nu,x}(\beta_0)}\d v=\frac{\{\sqrt{\pi}+O(\sqrt{\varepsilon_x})\}\Phi(2t)}{\sqrt{2\log_2 x}}.
	\end{equation}
	En remarquant que $h_{\nu,x}(\beta_0)=\sqrt{2/\pi}$, nous déduisons des estimations \eqref{eq:sPnu:inter} et \eqref{eq:integral:rapportgnux:final} que
	\begin{equation}\label{eq:eval:cPnu:mediumt}
		\sS_{\nu}(x,t)=\Phi(2t)+O\big(\sqrt{\varepsilon_x}\big).
	\end{equation}
	Il reste à évaluer le terme d'erreur $\sE_\nu(x,t)$ défini en \eqref{eq:defcPnu;cEnu}. Puisque $p\in\sP_{x,t}^-\Rightarrow\beta_p\leq\tfrac12-\eta_x$, nous sommes en mesure d'appliquer la majoration \eqref{eq:majo:Mnu:uniform}. Nous obtenons ainsi
	\begin{equation}\label{eq:eval:cEnu:mediumt}
		\sE_\nu(x,t)=\frac1x\sum_{p\in\sP_{x,t}^-}M_\nu(x,p)\ll\frac{1}{(\log_2 x)^{5/2}}\sum_{p\in\sP_{x,t}^-}\frac1p\ll\varepsilon_x^{3/2},
	\end{equation}
	d'après le théorème de Mertens. En regroupant les estimations \eqref{eq:eval:cPnu:mediumt} et \eqref{eq:eval:cEnu:mediumt} nous obtenons bien \eqref{eq:gaussiandistrib:uniform}.
	\par
	Supposons désormais que $t\leq-\sqrt{\log_3 x}$. Rappelons dans ce cas que $\sP_{x,t}^+=\varnothing$, $\sP_{x,t}^-=\sP_{x,t}$ et, par conséquent, $\cA_\nu(x,t)=\sE_\nu(x,t)$. D'une part, à l'aide d'une nouvelle utilisation de \eqref{eq:majo:Mnu:uniform} et du théorème de Mertens, nous avons,
	\[\cA_\nu(x,t)=\frac1x\sum_{p<\exp(\e^{\lambda(x,t)})}M_{\nu}(x,p)\leq\frac1x\sum_{\substack{p\leq x\\ \beta_p<1/2-\eta_x}}M_{\nu}(x,p)\ll\varepsilon_x^{3/2},\]
	et d'autre part, d'après \eqref{eq:majo:gaussian:tail},
	\[\Phi(-2\sqrt{\log_3 x})\ll\varepsilon_x^2.\]
	Les deux membres de \eqref{eq:gaussiandistrib:uniform} sont donc absorbés dans le terme d'erreur $O(\sqrt{\varepsilon_x})$.\par
	Il reste à traiter le cas $t\geq\sqrt{\log_3 x}$. Nous pouvons écrire
	\begin{equation}\label{eq:cAnu:bigt:decomp}
		\begin{aligned}
			\cA_{\nu}(x,t)&=\frac1x\bigg\{\sum_{p\leq\exp(\e^{\lambda(x,\sqrt{\log_3 x})})}M_{\nu}(x,p)+\sum_{\exp(\e^{\lambda(x,\sqrt{\log_3 x})})<p\leq\exp(\e^{\lambda(x,t)})}M_{\nu}(x,p)\bigg\}\\
			&=:\cA_\nu\big(x,\sqrt{\log_3 x}\big)+E_\nu(x,t).
		\end{aligned}
	\end{equation}
	Nous avons déjà démontré que
	\begin{equation}\label{eq:Anu:bigt:main}
		\cA_{\nu}\big(x,\sqrt{\log_3 x}\big)=\Phi(2\sqrt{\log_3 x})+O(\sqrt{\varepsilon_x})=1+O(\sqrt{\varepsilon_x}).
	\end{equation}
	Il nous reste à évaluer le terme d'erreur $E_\nu(x,t)$. Puisque $p\in[\exp(\e^{\lambda(x,\sqrt{\log_3 x})}),\exp(\e^{\lambda(x,t)})]$, nous avons $\tfrac12+\eta_x<\beta_p\leq\tfrac12+t\sqrt{\varepsilon_x}$. Nous pouvons appliquer une nouvelle fois la majoration \eqref{eq:majo:Mnu:uniform} et obtenir ainsi
	\begin{equation}\label{eq:Anu:bigt:error}
		E_\nu(x,t)\ll\frac1{(\log_2 x)^{5/2}}\sum_{\exp(\e^{\lambda(x,\sqrt{\log_3 x})})<p\leq\exp(\e^{\lambda(x,t)})}\frac1p\ll\varepsilon_x^{3/2}.
	\end{equation}
	En regroupant les estimations \eqref{eq:Anu:bigt:main} et \eqref{eq:Anu:bigt:error}, nous obtenons, d'après \eqref{eq:cAnu:bigt:decomp},
	\[\cA_\nu(x,t)=1+O(\sqrt{\varepsilon_x}).\]
	Enfin, puisque
	\[\Phi(2t)=1+O(\e^{-2t^2})=1+O(\sqrt{\varepsilon_x})\quad (t\geq\sqrt{\log_3 x}),\]
	l'expression $1+O(\sqrt{\varepsilon_x})$ constitue une formule asymptotique pour chacun des deux membres de~\eqref{eq:gaussiandistrib:uniform}. \par 
	Il reste à justifier l’optimalité du terme d’erreur de \eqref{eq:gaussiandistrib:uniform}. Remarquons d’emblée
que le terme d’ordre $\sqrt{\varepsilon_x}$ est issu des formules \eqref{eq:Mp:pomega}, \eqref{eq:int:rapportgnux:error1}, \eqref{eq:int:rapportgnux:error2}, tous les autres termes d’erreur sous-jacents étant $o(\sqrt{\varepsilon_x})$. L’étude menée dans \cite{bib:papier1} de la valeur moyenne du facteur premier médian fournit une expression explicite du terme d’erreur dans la formule \eqref{eq:Mp:pomega}, qui est donc optimal et $o(\sqrt{\varepsilon_x})$. Par ailleurs, un développement de Taylor à l’ordre 6 permet de préciser la formule \eqref{eq:taylor:hnu} sous la forme
	\[\frac{h_{\nu,x}(\beta_0+v)}{h_{\nu,x}(\beta_0)}=\{1+\tau v-2v^4\log_2 x+O(v^2+v^6\log_2 x)\}\e^{-2v^2\log_2 x}\quad(x\geq 3,\, v\ll 1),\]
	où $\tau$ est une constante absolue non nulle. Un calcul de routine montre alors que, pour $t=0$,
	\[\sqrt{\log_2 x}\int_{-\eta_x}^0\{\tau v-2v^4\log_2 x\}\e^{-2v^2\log_2 x}\d v=-\frac{\tau+o(1)}{4\sqrt{\log_2 x}}\asymp\sqrt{\varepsilon_x}.\]
	Cela termine la démonstration.\qed \bigskip
	
	\noindent{\it Remerciements.} L'auteur tient à remercier chaleureusement le professeur Gérald Tenenbaum pour l'ensemble de ses conseils et remarques avisés ainsi que pour ses relectures attentives durant la réalisation de ce travail.



\end{document}